\documentclass[a4j,12pt]{article}
\usepackage{amsmath}
 \usepackage{amssymb}
 \usepackage{amsthm}
\numberwithin{equation}{section}
 \theoremstyle{definition}
 \newtheorem{example}{Example}[section]
 
 \theoremstyle{plain}
 \newtheorem{lemma}[example]{Lemma}
 \newtheorem{proposition}[example]{Proposition}
 \newtheorem{theorem}[example]{Theorem}
 \theoremstyle{remark}
 \newtheorem{remark}[example]{Remark}
 \begin{document}
 \title{On the higher derivatives of $Z(t)$ associated with the Riemann Zeta-Function }
 \author{Kaneaki Matsuoka}
 \date{}
 \maketitle
 \begin{section}
 {INTRODUCTION}
 \end{section}
 Let $s=\sigma+it$ be a complex variable and $\zeta(s)$ the Riemann zeta-function. The functional equation for the Riemann zeta-function is 
$$h(s)\zeta(s)=h(1-s)\zeta(1-s),$$
where $h(s)=\pi^{-s/2}\Gamma(s/2)$. Define 
$$Z(t)=e^{i\theta(t)}\zeta(1/2+it),$$
where $\theta(t)=\arg h(1/2+it)$. From the functional equation it follows that $Z(t)$ is real and we can easily see that zeros of $Z(t)$ coincide with those of $\zeta(1/2+it)$. These properties make it possible to investigate the zeros of the Riemann zeta-function on the critical line. Using a function similar to $Z(t)$ Hardy \cite{Hardy} first proved that there are infinitely many zeros on the critical line and Hardy and Littlewood \cite{HL} showed that the number of zeros on the line segment from $1/2$ to $1/2+iT$ is $\gg T$. Siegel \cite{Siegel} showed that the number of those is  $>3e^{-3/2}T/8\pi+o(T)$. He defined $Z(t)$ and derived the Riemann-Siegel formula from the manuscript of Riemann, which was the essential part of his proof. Later A. Selberg \cite{Selberg} improved the bounds to $\gg T\log T$ and recently H. Bui, B. Conrey and M. Young \cite{Conrey} showed that more than 41 $\%$ of the zeros of the Riemann zeta-function are on the critical line. The Riemann-Siegel formula plays an important role in the investigation of the behavior of the Riemann zeta-function on the critical line as well as the calculation of the number of the complex zeros of the Riemann zeta-function. We sometimes call the function $Z(t)$ the Hardy function or the Riemann-Siegel function because of the above reason. \\
\indent It is well known that under the assumption of the Riemann hypothesis $Z'(t)$ has exactly one zero between consecutive zeros of $Z(t)$ (see Edwards [3, p.176]). R. J. Anderson \cite{Anderson} showed the same relationship between zeros of $Z'(t)$ and those of $Z''(t)$. K. Matsumoto and Y. Tanigawa \cite{Matsumoto} studied the number of zeros of the higher derivatives of $Z(t)$. They showed that under the assumption of the Riemann hypothesis the number of zeros of $Z^{(n)}(t)$ in the interval $(0,T)$ is $T/2\pi\log T/2\pi-T/2\pi+O(\log T)$, where $n$ is any positive integer and the implied constant depends on $n$. From this result we find that the same type of relationship as above is valid between $Z^{(n)}(t)$ and $Z^{(n+1)}(t)$ in almost cases except for $O(\log t)$ terms. In this paper we will prove the following theorem.\\
\begin{theorem}
If the Riemann hypothesis is true then for any positive integer $n$ there exists a $t_{n}>0$ such that for $t>t_{n}$ the function $Z^{(n+1)}(t)$ has exactly one zero between consecutive zeros of $Z^{(n)}(t)$.
\end{theorem}
The case $n=1$ is the result of Anderson [1, Theorem 3]. R. J. Anderson \cite{Anderson} constructed and studied the meromorphic function $\eta(s)$. K. Matsumoto and Y. Tanigawa \cite{Matsumoto} introduced a function $\eta_{n}(s)$ which is a generalization of Anderson's $\eta(s)$. These functions played an important role to show their results.  We will define the function $g_{n}(s)$ which has properties similar to those of Anderson's $\eta(s)$ in Section 2 and this section will be the most essential part of our proof. Theorem 1.1 will be proved in the last section after preparing some auxiliary results in Sections 3-5. These have been inspired by the proof of Anderson \cite{Anderson}.
\begin{section}
{THE DEFINITION OF FUNCTIONS}
\end{section}
Let $\chi(s)=h(1-s)/h(s)$ and $\omega(s)=(\chi'/\chi)(s)$. We see that 
\begin{equation}
\omega(1/2+it)=-2\theta'(t)\label{eq:ome1},
\end{equation}
and
\begin{equation}
\omega(1-s)=\omega(s)\label{eq:ome2}.
\end{equation}
 Now let $f_{0}(s)=\zeta(s)$, and we define $f_{n}(s)$ for $n\geq 1$ recursively by
\begin{equation}
f_{n+1}(s)=f_{n}'(s)-\frac{1}{2}\omega(s)f_{n}(s)\hspace{10mm}(n\geq 0)\label{eq:fdef}.
\end{equation}
Let $h_{0}(s)=1$, and we define $h_{n}(s)$ for $n\geq 1$ recursively by
\begin{equation}
h_{n+1}(s)=h_{n}'(s)-\frac{1}{2}\omega(s)h_{n}(s)\hspace{10mm}(n\geq 0)\label{eq:hdef}.
\end{equation}
We denote $g_{n}(s)$ by $f_{n}(s)/h_{n}(s)$. 
\begin{proposition}
For any non-negative integer $n$, we have
\begin{equation}
Z^{(n)}(t)=i^{n}f_{n}\Bigl(\frac{1}{2}+it\Bigl)e^{i\theta(t)}\label{eq:proz}.
\end{equation}
\end{proposition}
\begin{proof}
The case $n=0$ is the definition of $Z(t)$. Assume that (\ref{eq:proz}) is valid for $n$. Then
  $$Z^{(n+1)}(t)=\Bigl(i^{n+1}f_{n}'\Bigl(\frac{1}{2}+it\Bigl)+i^{n+1}\theta'(t)f_{n}\Bigl(\frac{1}{2}+it\Bigl)\Bigl)e^{i\theta(t)}.$$
  From (\ref{eq:ome1}), we find that (\ref{eq:proz}) is valid for $n+1$. Hence the result follows.
\end{proof}
Matsumoto and Tanigawa \cite{Matsumoto} defined a meromorphic function $\eta_{n}(s)$ which has the property
$$Z^{(n)}(t)=i^{n}\theta'(t)\eta_{n}(1/2+it)e^{i\theta(t)}.$$
From (\ref{eq:ome1}) and (\ref{eq:proz}) we have $f_{n}(s)=-\omega(s)\eta_{n}(s)/2$.
\begin{proposition}
For any non-negative integer $n$, we have
\begin{equation}
\chi(s)f_{n}(1-s)=(-1)^{n}f_{n}(s)\label{eq:funeq}.
\end{equation}
\end{proposition}
\begin{proof}
The case $n=0$ is nothing but the functional equation for the Riemann zeta-function. From (\ref{eq:ome2}) and the functional equation for $\eta_{n}(s)$ (see Matsumoto and Tanigawa [6, Proposition 2]), we obtain the result.
\end{proof}
\begin{remark}
From (\ref{eq:ome2}) and the definition of $g_{n}(s)$, if $n=1$ the formula which is of the same form as (\ref{eq:funeq}) but with replacing $f_{1}(s)$ by $g_{1}(s)$ is also valid. This is the functional equation for $\eta(s)$ (see Anderson \cite{Anderson}). But for $n\geq 2$ we can not replace $f_{n}(s)$ by $g_{n}(s)$ in (\ref{eq:funeq}).
\end{remark} 
From (\ref{eq:fdef}) we see that $f_{n}(s)$ can be expressed as
 \begin{equation}
 f_{n}(s)=\sum_{k=0}^{n}a_{n,k}(s)\zeta^{(k)}(s),\label{eq:f1}
 \end{equation}
 where $a_{n,k}(s)$ is a polynomial in the variables $\omega(s),\omega'(s),\cdots , \omega^{(n)}(s)$ with constant coefficients and we denote
 \begin{equation}
 a_{n,k}(s)=\sum_{h=0}^{n}c_{n,k,h}(s)\omega^{h}(s) ,\label{eq:a1}
 \end{equation} 
 where $c_{n,k,h}(s)$ is a polynomial in the variables $\omega'(s),\omega''(s),\cdots , \omega^{(n)}(s)$ with constant coefficients. It is easy to see that $a_{n,0}(s)=h_{n}(s)$ and hence we have
 $$g_{n}(s)=\zeta(s)+\sum_{k=1}^{n}\frac{a_{n,k}(s)}{h_{n}(s)}\zeta^{(k)}(s).$$
 We stress that the coefficient of $\zeta(s)$ is 1, which enables us to make use of a method in the theory of the Riemann zeta-function.
\begin{section}
{BASIC PROPERTIES OF $f_{n}(s),g_{n}(s)$ AND $h_{n}(s)$}
\end{section}
We denote by $n$ a positive integer and Landau's symbol $O$ depends on $n$. It is well known that 
\begin{equation}
\chi(s)=2^{s}\pi^{s-1}\sin(\pi s/2)\Gamma(1-s),\label{eq:chi1}
\end{equation}
 and it follows that 
\begin{equation}
\omega(s)=\log(2\pi)+\frac{\pi}{2}\tan\Bigl(\frac{\pi s}{2}\Bigl)-\frac{\Gamma'}{\Gamma}(s)\label{eq:ome3}.
\end{equation}
It is known that
      \begin{equation}
      \log \Gamma(s)=\left(s-\frac{1}{2}\right)\log s-s+\frac{1}{12s}-\int_{0}^{\infty}\frac{P(x)}{(s+x)^{3}}dx,\label{eq:x8}
      \end{equation}
       where $P(x)$ is a certain periodic function (see Edwards [3, p.109]). Hence we have
\begin{equation}
\frac{\Gamma'}{\Gamma}(s)=\log |s|+O(1)\hspace{10mm}(\sigma>1/4),\label{eq:gam1}
\end{equation}
and
\begin{equation}
\frac{d^{n}}{ds^{n}}\frac{\Gamma'}{\Gamma}(s)=O(|s|^{-n})\hspace{10mm}(n\geq 1, \sigma>1/4).\label{eq:gam3}
\end{equation}
Define the set $D$ by removing all small circles whose centers are odd positive integers and  even non-positive integers from the complex plane. We denote $D_{1}$ by $\mathbb{C}-D$.
\begin{lemma}
In the region $\{s\in D| \sigma>1/4, t>0\}$ we have
\begin{equation}
\tan s=i+O(e^{-2t}),\label{eq:tan1}
\end{equation}
and
\begin{equation}
\frac{d^{n}}{ds^{n}}\tan s=O(e^{-2t})\hspace{10mm}(n\geq 1).\label{eq:tan2}
\end{equation}
\end{lemma}
\begin{proof}
Since
$$\tan s=-i\Bigl(-1+\frac{2e^{2i\sigma -2t}}{e^{2i \sigma -2t}+1}\Bigl),$$
(\ref{eq:tan1}) immediately follows. Next we show (\ref{eq:tan2}). Since $\tan'(x)=\tan^{2}(x)+1$, we have $\tan'(s)=O(e^{-2t})$ and
\begin{equation}
\tan''(x)=2\tan(x)\tan'(x).\label{eq:tan3}
\end{equation} 
So the case $n=2$ follows. Assume that  (\ref{eq:tan2}) is valid for $n+2$. Differentiating (\ref{eq:tan3}) $n+1$ times we see that
\begin{align*}
\tan^{(n+3)}(x)&=2\sum_{k=0}^{n+1}\binom{n+1}{k}\tan^{(k)}(x)\tan^{(n-k+2)}(x)\\
&=2\tan^{(n+2)}(x)\tan(x)+2\sum_{k=1}^{n+1}\binom{n+1}{k}\tan^{(k)}(x)\tan^{(n-k+2)}(x).
\end{align*}
We find that (\ref{eq:tan2}) is valid for $n + 3$. This proves the lemma.
\end{proof}
\begin{lemma}
For $s\in D$, we have
 \begin{equation}
 \omega(s)=-\log |s|+O(1), \label{eq:ome4}
 \end{equation}
 and
 \begin{equation}
 \omega^{(n)}(s)=O(1)\hspace{10mm}(n\geq 1).\label{eq:ome5}
 \end{equation}
 \end{lemma}
 \begin{proof}
 From  the previous lemma and (\ref{eq:ome3}), we can prove the lemma if $s$ is in $ D\cap\{s|\sigma>1/4\}$. But considering equation (\ref{eq:ome2}) the lemma follows.
 \end{proof}
  If $k\neq 0$ and $n\geq 1$ then $c_{n,k,n}(s)=0$, hence with Lemma 3.2 if $s\in D$ we have
\begin{equation}
a_{n,k}(s)=O((\log|s|)^{n-1})\hspace{10mm}(k\neq 0, n\geq 1)\label{eq:a2}.
\end{equation}
  From (\ref{eq:hdef}) if $n\geq 1$ then $c_{n,0,n}(s)=(-1)^{n}/2^{n}$, hence with Lemma 3.2 and (\ref{eq:a1}) if $s\in D$ we have
 \begin{equation}
 h_{n}(s)=a_{n,0}(s)=\left(\frac{\log|s|}{2}\right)^{n}+O((\log|s|)^{n-1})\hspace{10mm}(n\geq 1)\label{eq:h1}.
 \end{equation}
 \begin{lemma}
 We have
 $$\zeta(s)=1+O(2^{-\sigma})\hspace{10mm}(\sigma>2),$$
 and
 $$\zeta^{(n)}(s)=O(2^{-\sigma})\hspace{10mm}(n\geq 1, \sigma>2).$$
 \end{lemma}
 \begin{proof}
 From the definition of the Riemann zeta-function we have
 $$\zeta(s)=1+\sum_{k=2}^{\infty}\frac{1}{k^{s}},$$
 and differentiating it $n$ times we get
 $$|\zeta^{(n)}(s)|\leq\sum_{k=2}^{\infty}\frac{(\log k)^{n}}{k^{\sigma}}.$$
 Let $f(x)=x^{-\sigma}(\log x)^{n}$. Since $f'(x)=x^{-\sigma-1}(\log x)^{n-1}(n-\sigma\log x)$, if $\sigma\geq n/\log 2$ then $f(x)$ is decreasing on $x\geq 2$. Hence with
 $$\int_{2}^{\infty}x^{-\sigma}(\log x)^{n}dx=\frac{2^{-\sigma+1}}{\sigma-1}(\log 2)^{n}+\frac{n}{\sigma-1}\int_{2}^{\infty}x^{-\sigma}(\log x)^{n-1}dx,$$
 we obtain the result.
 \end{proof}
  From (\ref{eq:f1}), (\ref{eq:a1}), (\ref{eq:a2}), (\ref{eq:h1}) and Lemma 3.3 , if $s\in D\cap\{s|\sigma>2\}$ we have
  \begin{align}
  f_{n}(s)&=\left(\left(\frac{\log |s|}{2}\right)^{n}+O((\log|s|)^{n-1})\right)(1+O(2^{-\sigma}))+\sum_{k=1}^{n}O\left(\frac{(\log|s|)^{n-1}}{2^{\sigma}}\right)\notag\\
  &=(1+O(2^{-\sigma}))\left(\frac{\log |s|}{2}\right)^{n}+O((\log |s|)^{n-1}).\label{eq:f2}
  \end{align}
    \begin{section}
  {ZEROS AND POLES OF $f_{n}(s), g_{n}(s), h_{n}(s)$}
  \end{section}
 From (\ref{eq:chi1}) we have 
  \begin{lemma}
  (Matsumoto and Tanigawa \cite{Matsumoto} Lemma 1) The poles of $\omega(s)$ are all simple, and are located at $1, 3, 5, \cdots$ (with residue $-1$) and at $0, -2, -4, \cdots$ (with residue $1$).
  \end{lemma}
   First we investigate the poles and zeros of $h_{n}(s)$. 
   \begin{lemma}
   The function $h_{n}(s)$ has poles of order $n$ which are located only at $1, 3, 5, \cdots, 0, -2, -4, \cdots.$ 
   \end{lemma}
   \begin{proof}
   The case $n=1$ is the previous lemma. Assume that this lemma is valid for $n$. Let $a$ be a pole of $h_{n}(s)$. We expand $h_{n}(s)$ in a Laurent series of powers of $s-a$. Thus $h_{n}(s)=c_{n}/(s-a)^{n}+\cdots $, where $c_{n}$ does not vanish. From Lemma 4.1 and (\ref{eq:hdef}) we have 
   $$h_{n+1}(s)=\frac{-nc_{n}\pm\frac{1}{2}c_{n}}{(s-a)^{n+1}}+\cdots.$$
   If $a$ is positive we take plus and $a$ is not positive we take minus. Since $-nc_{n}\pm c_{n}/2$ does not vanish, this lemma is valid for $n+1$. This proves the lemma.
   \end{proof}
   It is difficult to determine the location of zeros of $h_{n}(s)$ exactly but for large $|s|$ we roughly know the location.
   \begin{lemma}
   Let $2m$ be a sufficiently large even integer. In the region $\{s| \sigma\geq2m\}\cup\{s| \sigma\leq1-2m\}$ zeros of $h_{n}(s)$ are all located in $D_{1}$ and the number of those in a circle is $n$. Let $T$ be sufficiently large. In the region $\{s|1-2m<\sigma<2m, |t|>T\}$ there exists no zero of $h_{n}(s)$.
   \end{lemma} 
   \begin{proof}
   From (\ref{eq:h1}) if $|s|$ is sufficiently large $\Re h_{n}(s)$ is positive in $D$. By the argument principle and the previous lemma the result follows.
   \end{proof}
   Next we investigate the poles and zeros of $f_{n}(s)$.
   \begin{lemma}
   The function $f_{n}(s)$ has poles of order $n$ located at $0, 3, 5, 7, \cdots$ and those of order $n+1$ located at $1$ and those of order $n-1$ located at $-2, -4, -6\cdots.$
   \end{lemma}
   \begin{proof}
   From (\ref{eq:fdef}) the case $k=1, 3, 5\cdots$ is proved in the same way as in Lemma 4.2. The case $k=0, -2, -4\cdots$ is proved from (\ref{eq:funeq}).
   \end{proof}
   \begin{lemma}
   Let $2m$ be a sufficiently large even integer. In the region $\{s|\sigma\geq2m\}\cup\{s|\sigma\leq1-2m\}$ zeros of $f_{n}(s)$ are located in $D_{1}$ and the number of those in a circle is $n$. 
   \end{lemma}
   \begin{proof}
   From (\ref{eq:f2}) if $m$ is sufficiently large $\Re f_{n}(s)$ is positive in the region $\{s|\sigma\geq 2m\}\cap D$ hence the lemma is proved in the same way as in Lemma 4.1. From (\ref{eq:funeq}) the lemma is also proved in the region $\{s|\sigma\leq 1-2m\}\cap D$.
   \end{proof}
 
  From the previous lemmas we know the location of the zeros and poles of $g_{n}(s)$. For each circle $B$ included in $D_{1}$, let $N_{0}(B)$ (resp. $N_{\infty}(B)$) be the number of zeros (resp. poles) in $B$.
  \begin{lemma}
  Let $T$ and $m$ be large. In the region $\{s| 1-2m<\sigma<2m, |t|>T\}$ there exists no pole of $g_{n}(s)$. In the region $\{s|\sigma\geq2m\}$ zeros and poles of $g_{n}(s)$ are all located in $D_{1}$ and the number of zeros in a circle $B$ is at most $n$ and $N_{0}(B)=N_{\infty}(B)$. In the region $\{s|\sigma\leq 1-2m\}$ zeros and poles of $g_{n}(s)$ are located in $D_{1}$ and the number of zeros in a circle $B$ is at most $n+1$ and $N_{0}(B)=N_{\infty}(B)+1$.
  \end{lemma}
  Let $m=m(n)$ be a sufficiently large positive integer.
 We define $N_{g_{n}}(T)$ by the number of zeros of $g_{n}(s)$ with $-2m+1<\sigma<2m$ and $0<t<T$. From Lemma 4.3 and the proof of Theorem 1 and Theorem 2 in Matsumoto and Tanigawa \cite{Matsumoto} we have the following results.
   \begin{proposition}
   We have
   $$N_{g_{n}}(T)=\frac{T}{2\pi}\log\frac{T}{2\pi}-\frac{T}{2\pi}+O(\log T).$$
  \end{proposition}

\begin{proposition}
On the Riemann hypothesis if $T$ is large, zeros of $g_{n}(s)$ in the region  $\{s | 1-2m<\sigma<2m, |t|>T\}$ are on the critical line $\sigma=1/2$.
\end{proposition}
\begin{section}
{LEMMAS FOR THE PROOF OF THE THEOREM}
\end{section}
\begin{lemma}
  Let $T$ be large. There exists a positive number $A$ such that in the region $\{s| -2m+1<\sigma<2m, t>T\}$
  $$g_{n}(s)=O(t^{A}).$$
  \end{lemma}
  \begin{proof}
  Let $M\geq 2$ and $K$ be positive integers. It is known that 
  \begin{align*}
  \zeta(s)&=\sum_{n=1}^{M}\frac{1}{n^{s}}+\frac{M^{1-s}}{s-1}-\frac{M^{-s}}{2}+\sum_{k=1}^{K}\frac{B_{2k}}{(2k)!}s(s+1)\cdots (s+2k-2)M^{1-s-2k}\\
  &-\frac{s(s+1)\cdots(s+2k)}{(2K+1)!}\int_{M}^{\infty}B_{2K+1}(x-[x])x^{-s-2K-1}dx,
  \end{align*}
  where $B_{n}$ is the $n$-th Bernoulli number and $B_{n}(x)$ is the $n$-th Bernoulli polynomial (see Edwards [3, p.114]). Hence with Lemma 3.2 the result follows. 
  \end{proof}
\begin{lemma}
For $s\in D\cap\{s|\sigma>2m\}$ we have
 $$\frac{g'_{n}}{g_{n}}(\sigma+it)=O(1).$$
\end{lemma}
\begin{proof}
   From (\ref{eq:f1}), (\ref{eq:a1}), (\ref{eq:a2}), (\ref{eq:h1}) and Lemma 3.3 for the region  $\{s|\sigma\geq 2m\}\cap D$ we have
 \begin{align}
 g_{n}(s)&=\zeta(s)+\sum_{k=1}^{n}\frac{a_{n,k}(s)}{h_{n}(s)}\zeta^{(k)}(s)\label{eq:g2}\\
 &=1+O(2^{-\sigma})+\sum_{k=1}^{n}\frac{O((\log|s|)^{n-1}2^{-\sigma})}{\left(\frac{\log|s|}{2}\right)^{n}+O((\log|s|)^{n-1})}\notag \\ 
 &=1+O(2^{-\sigma}).\label{eq:g1}
 \end{align}
Differentiating (\ref{eq:g2}) we have
\begin{equation}
g'_{n}(s)=\zeta'(s)+\sum_{k=1}^{n}\left(\left(\frac{a_{n,k}'(s)}{h_{n}(s)}-\frac{a_{n,k}(s)h_{n}'(s)}{h_{n}(s)^{2}}\right)\zeta^{(k)}(s)+\frac{a_{n,k}(s)}{h_{n}(s)}\zeta^{(k+1)}(s)\right).\label{eq:9}
\end{equation}
From (\ref{eq:a1}), (\ref{eq:ome4}), (\ref{eq:ome5}), if $k\neq 0$ we have $a_{n,k}^{(l)}(s)=O((\log |s|)^{n-1})$ for any positive integer $l$. Hence with (\ref{eq:a2}), (\ref{eq:h1}), (\ref{eq:9}) and Lemma 3.3 we have
\begin{align*}
g'_{n}(s)&=O(2^{-\sigma})+\frac{O(2^{-\sigma}(\log |s|)^{n-1})}{(\log|s|)^{n}+O((\log |s|)^{n-1})}+\frac{O(2^{-\sigma}(\log|s|)^{2n-2})}{(\log |s|)^{2n}+O((\log |s|)^{2n-1})}\\
&=O(2^{-\sigma}).
\end{align*}
Hence with (\ref{eq:g1}) we have
$$\frac{g'_{n}}{g_{n}}(\sigma+it)=\frac{O(2^{-\sigma})}{1+O(2^{-\sigma})}=O(1).$$
\end{proof}
\begin{lemma}
We have
$$N_{g_{n}}(T+1)-N_{g_{n}}(T)=O(\log T).$$
\end{lemma}
\begin{proof}
From Proposition 4.7 we have
\begin{align*}
N_{g_{n}}(T+1)-N_{g_{n}}(T)&=\Bigl(\frac{T+1}{2\pi}\Bigl)\Bigl(\log\frac{T}{2\pi}+O\Bigl(\frac{1}{T}\Bigl)\Bigl)-\frac{T}{2\pi}\log\frac{T}{2\pi}+O(\log T)\\
&=O(\log T).
\end{align*}
\end{proof}
\begin{lemma}(Titchmarsh [9, p.56] LEMMA $\alpha$)
If $f(s)$ is regular, and
$$\left|\frac{f(s)}{f(s_{0})}\right|<e^{M}\hspace{10mm}(M>1)$$
in the circle $|s-s_{0}|\leq r$, then
$$\left|\frac{f'(s)}{f(s)}-\sum_{\rho}\frac{1}{s-\rho}\right|<\frac{AM}{r}\hspace{10mm}(|s-s_{0}|\leq\frac{1}{4}r),$$
where $\rho$ runs through the zeros of $f(s)$ such that $|\rho-s_{0}|\leq \frac{1}{2}r$ and $A$ is a positive constant.
\end{lemma}
\begin{lemma}
For large $t$ and $1-2m\leq\sigma \leq 2m$ we have
$$\frac{g'_{n}}{g_{n}}(\sigma+it)=\sum_{|t-\gamma|<1}\frac{1}{s-\rho}+O(\log t),$$
where $\rho=\beta+i\gamma$ runs through the zeros of $g_{n}(s)$ such that $|t-\gamma|<1$.
\end{lemma}
\begin{proof}
Let $T$ be sufficiently large and $f(s)=g_{n}(s)$, $s_{0}=2m+1+iT$, $r=16m+4$ in the previous lemma. If $|s-s_{0}|\leq 16m+4$ then $-14m-3\leq \sigma \leq 18m+5$, $T-16m-4\leq t \leq T+16m+4$. There exists a constant $A$ such that $g_{n}(\sigma+it)=O(t^{A})$ uniformly in the region $\{s| -14m-3\leq \sigma \leq 18m+5, t>T-16m-4\}$ by Lemma 5.1 and we have $g_{n}(2m+1+it)=1+O(4^{-m})$ by (\ref{eq:g1}). Hence with the previous lemma in the region $|s-s_{0}|\leq 4m+1$ we have
$$\frac{g'_{n}}{g_{n}}(\sigma+it)=\sum_{\rho}\frac{1}{s-\rho}+O(\log T),$$
where the summation is over the zeros of $g_{n}(s)$ such that $|\rho-s_{0}|\leq 8m+2$. In particular if $t=T$ we have
$$\frac{g'_{n}}{g_{n}}(\sigma+it)=\sum_{\rho}\frac{1}{s-\rho}+O(\log t).$$
If $|\rho-s_{0}|\leq 8m+2$ but $|\gamma-t|\geq 1$ then $|s-\rho|\geq 1$ hence with Lemma 5.3 we have
$$\frac{g'_{n}}{g_{n}}(\sigma+it)=\sum_{|t-\gamma|<1}\frac{1}{s-\rho}+O(\log t),$$
uniformly for $|\sigma-2m-1|\leq 4m+1$.
\end{proof}
\begin{lemma}
There exists a sequence $\{T_{j}\}$ tending to infinity, such that  if $g_{n}(\beta+i\gamma)=0$ then $|\gamma-T_{j}|^{-1}=O(\log T_{j}).$
\end{lemma}
\begin{proof}
Let $j$ be a large positive integer. Suppose the rectangle defined by $-2m+1\leq\sigma\leq2m$, $j\leq t\leq j+1$ contains $N$ zeros of $g_{n}(s)$. Divide it to $N+1$ rectangles of width $1/N+1$. At least one of these contains no zero of $g_{n}(s)$. There is  a $T_{j}$ with $j<T_{j}<j+1$ such that $|\gamma-T_{j}|>1/2(N+1)$. From  Lemma 5.3 we have $|\gamma-T_{j}|^{-1}=O(N)=O(\log T_{j})$.
\end{proof}
\begin{lemma}
There exists a sequence $\{T_{j}\}$ tending to infinity, such that 
$$\frac{g_{n}'}{g_{n}}(\sigma+iT_{j})=O(\log^{2} T_{j}),$$
uniformly for $-2m+1\leq \sigma\leq 2m$.
\end{lemma}
\begin{proof}
Let $\{T_{j}\}$ be as in Lemma 5.6. If $s=\sigma+iT_{j}$ then $|s-\rho|^{-1}\leq|T_{j}-\gamma|^{-1}= O(\log T_{j})$. Since the number of zeros with $|\gamma-T_{j}|<1$ is $O(\log T_{j})$  so Lemma 5.5 implies
\begin{align*}
\frac{g_{n}'}{g_{n}}(\sigma+iT_{j})&=O(\log^{2}T_{j})+O(\log T_{j})\\
&=O(\log^{2} T_{j}).
\end{align*}
\end{proof}
\begin{section}
{PROOF OF THEOREM 1.1}
\end{section}
It is sufficient to prove that $Z^{(n+1)}/Z^{(n)}(t)$ is decreasing  for large $t$. Write $T_{j}$ as in Lemma 5.7, for $k\geq m$ let $R_{k}$ be the rectangle with corners at $1-2k\pm iT_{j}$, $2k\pm iT_{j}$. Let $G_{n}(w)=h(w)g_{n}(w)$ and $s=\sigma+it$, where $s$ is in $R_{k}$ and $t$ is not the ordinate of some zero or pole of $G_{n}(s)$.  We define $I$ by
 \begin{equation}
 I=\frac{1}{2\pi i}\int_{\partial R_{k}}\frac{G'_{n}}{G_{n}}(w)\frac{s}{w(s-w)}dw.\label{eq:i1}
 \end{equation}
From (\ref{eq:funeq}) we have
 $$\chi(s)h_{n}(1-s)g_{n}(1-s)=(-1)^{n}h_{n}(s)g_{n}(s),$$
 hence we get
 \begin{equation}
 \frac{h'_{n}}{h_{n}}(s)+\frac{g'_{n}}{g_{n}}(s)=-\frac{h'_{n}}{h_{n}}(1-s)-\frac{g'_{n}}{g_{n}}(1-s)+\omega(s).\label{eq:10}
 \end{equation}
 Write $h_{n}(s)$ as
 \begin{equation}
 h_{n}(s)=\sum_{k=0}^{n}b_{n,k}(s)\omega^{k}(s)=(-1)^{n}\frac{\omega^{n}(s)}{2^{n}}+\sum_{k=0}^{n-1}b_{n,k}(s)\omega^{k}(s),\label{eq:h4}
 \end{equation}
 where $b_{n,k}(s)=c_{n,0,k}(s)$. Differentiating it we have
  \begin{equation}
  h'_{n}(s)=\frac{(-1)^{n}}{2^{n}}n\omega'(s)\omega^{n-1}(s)+\sum_{k=0}^{n-1}b_{n,k}(s)k\omega^{k-1}(s)\omega'(s)+\sum_{k=0}^{n-1}b_{n,k}'(s)\omega^{k}(s).\label{eq:h2}
  \end{equation}
  Since $b_{n,k}(s)$ is a polynomial in the variables $\omega'(s),\omega''(s),\cdots , \omega^{(n)}(s)$ with constant coefficients, $b_{n,k}'(s)$ is also a polynomial in the variables $\omega'(s),\omega''(s),\cdots , \omega^{(n+1)}(s)$ with constant coefficients. Therefore we have 
  \begin{equation}
  b_{n,k}(s)=O(1),\label{eq:x2}
  \end{equation}
  and $b_{n,k}'(s)=O(1)$ in $D$, hence with (\ref{eq:ome4}), (\ref{eq:ome5}), (\ref{eq:h2}) it follows that $h_{n}'(s)=O((\log|s|)^{n-1})$ in $D$. With (\ref{eq:h1}) if $|s|$ is large we obtain
  \begin{equation}
  \frac{h'_{n}}{h_{n}}(s)=O(1) \label{eq:h3} 
  \end{equation}
  in $D$. From Lemma 5.2 and Lemma 5.7 we have
  \begin{equation}
  \frac{g_{n}'}{g_{n}}(\sigma+iT_{j})=O(\log^{2} T_{j})\label{eq:11}
  \end{equation}
  for $1-2m\leq \sigma\leq 2k$. Hence with (\ref{eq:ome4}), (\ref{eq:10}), (\ref{eq:h3}), we have
  \begin{equation}
  \frac{g'_{n}}{g_{n}}(\sigma+iT_{j})=O(\log^{2} (k+T_{j}))\label{eq:g3}
  \end{equation}
  for $1-2k\leq \sigma\leq 2k$. From the definition of $G_{n}(s)$ we have 
   \begin{equation}
   \frac{G'_{n}}{G_{n}}(s)=\frac{h'}{h}(s)+\frac{g'_{n}}{g_{n}}(s),\label{eq:12}
   \end{equation}
   hence with (\ref{eq:gam1}) and (\ref{eq:g3}) we have
   \begin{equation}
   \frac{G'_{n}}{G_{n}}(\sigma+iT_{j})=O(\log^{2}(k+T_{j}))\label{eq:13}
   \end{equation}
   for $1/4< \sigma\leq 2k$. 
      From (\ref{eq:10}) and (\ref{eq:12}) we have
   \begin{align}
   \frac{G'_{n}}{G_{n}}(1-s)&=-\frac{h'_{n}}{h_{n}}(s)-\frac{G'_{n}}{G_{n}}(s)+\frac{h'}{h}(s)-\frac{h'_{n}}{h_{n}}(1-s)+\frac{h'}{h}(1-s)+\omega(s)\notag \\
   &=-\frac{h'_{n}}{h_{n}}(s)-\frac{G'_{n}}{G_{n}}(s)-\frac{h'_{n}}{h_{n}}(1-s),
   \label{eq:x11}
   \end{align}
   hence with (\ref{eq:h3}), (\ref{eq:13}) we have
    \begin{equation}
   \frac{G'_{n}}{G_{n}}(\sigma+iT_{j})=O(\log^{2}(k+T_{j}))\label{eq:x12}
   \end{equation}
   for $1-2k\leq \sigma\leq 2k$.
   If $w$ is on the horizontal sides of $\partial R_{k}$ and $T_{j}$ is sufficiently larger than $t$, then $|w(s-w)|>T^{2}_{j}/2$.
   Hence with (\ref{eq:x12}) we have 
   \begin{equation}
   \frac{1}{2\pi i}\int\frac{G'_{n}}{G_{n}}(w)\frac{s}{w(s-w)}dw=O\left(\frac{k\log^{2}(k+T_{j})}{T_{j}^{2}}\right), \label{eq:14}
   \end{equation}
   where the path of integration is the horizontal sides of $\partial R_{k}$ and the implied constant depends on $s$. Let us consider the vertical sides of $I$. We can write the integrals on the vertical sides as
   \begin{equation}
   \frac{1}{2\pi i}\int_{2k-iT_{j}}^{2k+iT_{j}}\frac{G_{n}'}{G_{n}}(w)\frac{s}{w(s-w)}dw-\frac{1}{2\pi i}\int_{2k-iT_{j}}^{2k+iT_{j}}\frac{G_{n}'}{G_{n}}(1-w)\frac{s}{(1-w)(s-1+w)}dw.\label{eq:15}
   \end{equation}
   From (\ref{eq:gam1}), (\ref{eq:g1}), (\ref{eq:12}),  we have
   \begin{equation}
   \frac{G'_{n}}{G_{n}}(2k+iy)=O(\log(k+|y|)).\label{eq:16}
   \end{equation}
   Hence with (\ref{eq:h3}), (\ref{eq:x11}) we have
   \begin{equation}
   \frac{G'_{n}}{G_{n}}(1-2k+iy)=O(\log(k+|y|)).\label{eq:17}
   \end{equation}
   From (\ref{eq:16}), (\ref{eq:17}) if we write $w=2k+iy$ the integral in (\ref{eq:15}) are 
   $$O\left(\frac{\log(k+|y|)}{k^{2}+y^{2}}\right),$$
   where the implied constant depends on $s$.
   Since
   \begin{align*}
   \int_{-\infty}^{\infty}\frac{\log(k+|y|)}{k^{2}+y^{2}}dy&=\frac{2}{k}\int_{0}^{\infty}\frac{\log(k+kz)}{1+z^{2}}dz\\
   &=O(k^{-1}\log k),
   \end{align*}
    we can see that (\ref{eq:15}) is $O(k^{-1}\log k)$ and with (\ref{eq:14}) we have
   \begin{align}
   I=O(kT_{j}^{-2}\log^{2}(k+T_{j}))+O(k^{-1}\log k),\label{eq:i2}
   \end{align}
   where the implied constant depends on $s$. One can evaluate $I$ by the residue theorem. By Lemma 4.2 and Lemma 4.4 $w=0$ is a simple pole of $G_{n}(w)$, hence we have
   \begin{equation}
   I=-\frac{G_{n}'}{G_{n}}(s)+r_{0}+\sum_{r}\frac{s}{a_{r}(s-a_{r})}-\sum_{r}\frac{s}{b_{r}(s-b_{r})},\label{eq:i3}
   \end{equation}
   where $a_{r}$ runs through the  zeros of $G_{n}(w)$ in $R_{k}$, $b_{r}$ runs through the poles of $G_{n}(w)$ in $R_{k}$ except at $w=0$ and $r_{0}$ is the residue of the integral in (\ref{eq:i1}) at $w=0$. If we expand $s/w(s-w)$ in a Laurent series of powers of $w$ then the constant term is $1/s$.  Hence with (\ref{eq:i2}) as $j\rightarrow\infty$ and $k\rightarrow\infty$ in (\ref{eq:i3}) we have
   \begin{equation}
   \frac{G_{n}'}{G_{n}}(s)=-\frac{1}{s}+A+\sum_{r_{1}}\frac{s}{a_{r_{1}}(s-a_{r_{1}})}-\sum_{r_{1}}\frac{s}{b_{r_{1}}(s-b_{r_{1}})},\label{eq:x9}
   \end{equation}
   where $A$ is a constant, $a_{r_{1}}$ runs through all zeros of $G_{n}(w)$ and $b_{r_{1}}$ runs through the poles of $G_{n}(w)$ except at $w=0$. From Lemma 4.6 and Lemma 5.3 these sums are locally uniformly absolutely convergent since $\sum_{n=1}^{\infty}\log n/n^{2}$ is convergent. Let $s=1/2+it$ in (\ref{eq:x9}) and assume the Riemann hypothesis hereafter. From Proposition 4.8 if we differentiate (\ref{eq:x9}) with respect to $t$ then we have
  \begin{equation}
  i\frac{d}{dt}\frac{G_{n}'}{G_{n}}\left(\frac{1}{2}+it\right)=O(|t|^{-2})-\sum_{\gamma}\frac{1}{(t-\gamma)^{2}}+\sum_{r_{2}}\frac{1}{(\frac{1}{2}+it-a_{r_{2}})^{2}}-\sum_{r_{2}}\frac{1}{(\frac{1}{2}+it-b_{r_{2}})^{2}},\label{eq:x10}
  \end{equation}
  where $\gamma$ runs through the zeros of $G_{n}(w)$ on the critical line, $a_{r_{2}}$ runs through those in the region $D_{1}-\{w|1-2m<\Re w<2m\}$ and $b_{r_{2}}$ runs through the poles of $G_{n}(w)$ in the same region. From Lemma 4.6 we have 
  $$\sum_{r}\frac{1}{(\frac{1}{2}+it-a_{r})^{2}}\ll\sum_{k=m}^{\infty}\frac{n}{(t+2k+1)^{2}}\ll\int_{2m+1}^{\infty}\frac{dx}{(t+x)^{2}}\ll\frac{1}{|t|},$$
   similarly we have
   $$\sum_{r}\frac{1}{(\frac{1}{2}+it-b_{r})^{2}}=O(|t|^{-1}),$$
   hence with (\ref{eq:x10}) we have
   \begin{equation}
   i\frac{d}{dt}\frac{G_{n}'}{G_{n}}\left(\frac{1}{2}+it\right)=O(|t|^{-1})-\sum_{\gamma}\frac{1}{(t-\gamma)^{2}}.\label{eq:x5}
   \end{equation}
   Let 
   $$F(t)=\left|h\left(\frac{1}{2}+it\right)\right|.$$
   From Proposition 2.1 we have
    $$Z^{(n)}(t)h_{n}^{-1}\left(\frac{1}{2}+it\right)h\left(\frac{1}{2}+it\right)=i^{n}G_{n}\left(\frac{1}{2}+it\right)e^{i\theta(t)},$$
    hence with the definition of $F(t)$ we get
    \begin{equation}
    \frac{Z^{(n+1)}(t)}{Z^{(n)}(t)}=i\frac{G_{n}'}{G_{n}}\left(\frac{1}{2}+it\right)+i\frac{h_{n}'}{h_{n}}\left(\frac{1}{2}+it\right)-\frac{F'(t)}{F(t)}.\label{eq:x6}
    \end{equation}
    From (\ref{eq:ome3}), (\ref{eq:gam3}) and Lemma 3.1 we have
    \begin{equation}
    \omega^{(j)}(s)=O(|t|^{-j}),\label{eq:x1}
    \end{equation}
    for $D\cap \{s|-2m+1<\sigma<2m\}$ where $j$ is a positive integer. Since $b_{n,k}^{(j)}(s)$ is a polynomial in the variables $\omega'(s),\omega''(s),\cdots , \omega^{(n+j)}(s)$ with constant coefficients whose constant term vanishes, we have
    \begin{equation}
    b_{n,k}^{(j)}(s)=O(|t|^{-1}) \label{eq:x3}
    \end{equation}
     in the same region. From (\ref{eq:h1}) we have
     \begin{equation*}
     h_{n}(s)=\frac{(\log |t|)^{n}}{2^{n}}+O((\log |t|)^{n-1}),
     \end{equation*}
     and from (\ref{eq:ome4}), (\ref{eq:h4}), (\ref{eq:x2}), (\ref{eq:x1}), (\ref{eq:x3}) we have
     \begin{equation*}
     h_{n}^{(j)}(s)=O(|t|^{-1}(\log |t|)^{n-1})
     \end{equation*}
     in the same region. It follows that if $t$ is large then 
     \begin{equation}
     \frac{d}{ds}\left(\frac{h_{n}'(s)}{h_{n}(s)}\right)=\frac{h_{n}''(s)}{h_{n}(s)}-\frac{(h_{n}'(s))^{2}}{h_{n}(s)^{2}}=O\left((|t|\log|t|)^{-1}\right)\label{eq:x4}
     \end{equation}
      for $D\cap \{s|-2m+1<\sigma<2m\}$.
      It is known (see Edwards [3, p.177]) that
      \begin{equation}
      \frac{d}{dt}\frac{F'(t)}{F(t)}=O(|t|^{-2}).\label{eq:x7}
      \end{equation}
      If $0<\gamma<t$ then $0<t-\gamma<t$ so $(t-\gamma)^{2}<t^{2}$. Hence if $t$ is large then by (\ref{eq:x5}), (\ref{eq:x6}), (\ref{eq:x4}), (\ref{eq:x7}) we have
       \begin{align}
   \frac{d}{dt}\frac{Z^{(n+1)}(t)}{Z^{(n)}(t)}&=-\sum_{\gamma}\frac{1}{(t-\gamma)^{2}}+O(t^{-1})+O(t^{-1}(\log t)^{-1})+O(t^{-2})\notag \\
   &<-\sum_{0<\gamma<t}\frac{1}{(t-\gamma)^{2}}+At^{-1}\notag\\
   &<-t^{-2}N'_{g_{n}}(t)+At^{-1}\notag\\
   &=t^{-1}(A-t^{-1}N'_{g_{n}}(t)),\label{eq:last}
   \end{align}
   where $N'_{g_{n}}(T)$ is the number of zeros of $g_{n}(1/2+it)$ with $0<t<T$ and $A$ is a positive constant. From Proposition 4.7 and Proposition 4.8, (\ref{eq:last}) is negative for large $t$. This completes the proof.
 
 Address: {Graduate School of Mathematics, Nagoya University, Furocho, Chikusaku, Nagoya 464-8602, Japan}\\ \\
 Email address: {m10041v@math.nagoya-u.ac.jp}
 \end{document}